\documentclass[12pt,a4paper]{amsart}
\usepackage{mathrsfs}

\newtheorem{theo+}              {Theorem}           [section]
\newtheorem{prop+}  [theo+]     {Proposition}
\newtheorem{coro+}  [theo+]     {Corollary}
\newtheorem{lemm+}  [theo+]     {Lemma}
\newtheorem{exam+}  [theo+]     {Example}
\newtheorem{rema+}  [theo+]     {Remark}
\newtheorem{defi+}  [theo+]     {Definition}

\newenvironment{theorem}{\begin{theo+}}{\end{theo+}}
\newenvironment{proposition}{\begin{prop+}}{\end{prop+}}
\newenvironment{corollary}{\begin{coro+}}{\end{coro+}}
\newenvironment{lemma}{\begin{lemm+}}{\end{lemm+}}
\newenvironment{definition}{\begin{defi+}}{\end{defi+}}

\usepackage{amsthm}
\theoremstyle{plain} \theoremstyle{remark}
\newtheorem{remark}{Remark}
\newtheorem{example}{Example}

\newtheorem*{ack}{\bf Acknowledgments}

\def \r{\mbox{${\mathbb R}$}}
\def\E{/\kern-1.0em \equiv }
\def\gh{generalized harmonic morphism}

\author{Elsa Ghandour$^{*}$}
\address{\hskip-\parindent
Laboratoire de Math\'ematiques de Bretagne Atlantique UMR 6205 \\
Universit\'e de Bretagne Occidentale, % \newline
29238 Brest Cedex 3\\
France}
\email{Elsa.ghandour@univ-brest.fr}

\author{Ye-Lin Ou$^{**}$}

\address{\hskip-\parindent
Department of Mathematics, Texas A $\&$ M University-Commerce,
\newline Commerce,  TX 75429,  USA} 
\email{yelin.ou@tamuc.edu}
\thanks{$^{*}$The first author would like to thank the University of Brest and the University of Bretagne Loire for  scholarships which allowed her to visit  the Department of Mathematics at Texas A $\&$ M University-Commerce for two months in the Fall of 2017. She is also grateful to Texas A $\&$ M University-Commerce and the Department of Mathematics for the hospitality she received during her visit there where this work was done. 
\newline\indent $^{**}$The second author is supported by a grant from the Simons Foundation ($\#427231$, Ye-Lin Ou)}
\date{12/08/2017}
\begin{document}
\title[Generalized harmonic morphisms and biharmonic maps]{Generalized harmonic morphisms and horizontally weakly conformal biharmonic maps}


\maketitle

\maketitle
\section*{Abstract}
\begin{quote} 
{\footnotesize Harmonic morphisms are maps between Riemannian manifolds that pull back harmonic functions to harmonic functions. These maps are characterized  as horizontally weakly conformal harmonic maps and they have many interesting links and applications to several areas in mathematics (see the book \cite{BW1} by Baird and Wood for details). In this paper, we study generalized harmonic morphisms which are defined to be  maps between Riemannian manifolds that pull back  harmonic functions to  biharmonic functions. We obtain some characterizations of generalized harmonic morphisms into a Euclidean space and  give two methods of constructions that can be used to produce many examples of  generalized harmonic morphisms which are not  harmonic morphisms. We also give a complete classification of \gh s among the projections of a warped product space, which provides infinitely many examples of proper biharmonic Riemannian submersions and  conformal submersions from a warped product manifold.}
\end{quote}

\section{Introduction}

A harmonic morphism is a map $\phi: (M^m, g)\longrightarrow (N^n, h)$ between Riemannian  manifolds that preserves the solutions of the Laplace equation in the sense that it pulls back any local harmonic function on  $(N^n, h)$ to a local harmonic function on $(M^m, g)$. Such maps are characterized by Fuglede \cite{Fu} and Ishihara \cite{Is} independently as  harmonic maps which are also horizontally weakly conformal. Here, horizontally weakly conformal maps are generalizations of Riemannian submersions in the sense that at the point where $(\rm d \phi)_x\ne0$, $(\rm d \phi)_x$ preserves horizontal angles. This is equivalent to the existence of a function $\lambda$ on $M$ such that 
\begin{equation}
h( \rm d \phi_x(X), \rm d \phi_x(Y))=\lambda^2g(X, Y) 
\end{equation}
 for any horizontal vectors $X, Y$. We refer the readers to the book \cite{BW1} for a comprehensive account of the theory, applications and interesting links of harmonic morphisms.
 
Biharmonic morphisms, as a generalization of the notion of harmonic morphisms, were introduced and studied  in \cite{Ou1}, \cite{LO1} and \cite{LO2}.  These are maps between Riemannian manifolds which preserve the solutions of bi-Laplace equations in the sense that they pull back germs of biharmonic functions to germs of biharmonic functions. According to a characterization obtained in \cite{LO2}, a map between Riemannian manifolds is a biharmonic morphism if and only if it is a horizontally weakly conformal map which is also a biharmonic map, a $4$-harmonic map, and satisfies an additional equation. So biharmonic morphisms are a very restricted class of horizontally weakly conformal biharmonic maps. We would also like to point out that not every harmonic morphism is a biharmonic morphism though the  latter generalizes the notion of the former.

In this paper, we study maps between Riemannian manifolds that pull back local harmonic functions to local biharmonic functions. Such maps clearly include harmonic morphisms as a subclass since any harmonic morphism pulls back a harmonic function to a harmonic function which is alway a biharmonic function. So, we call this class of maps generalized harmonic morphisms. We give two characterizations of \gh s into Euclidean spaces and two methods of constructions to produce \gh s by using direct sum of given \gh s or by composing a given \gh\,   with a harmonic or a biharmonic morphism. Many examples of \gh s which are not harmonic morphisms are given. We also give a complete classification of \gh s among the projections of a warped product space, which provides infinitely many examples of proper biharmonic Riemannian submersions and  conformal submersions from a warped product manifold. Our study shows that  the \gh s are not only  a natural generalization of harmonic morphisms (see, e.g., Proposition \ref{R3} and Corollary \ref{CR3}) but  also a useful tool to construct  infinitely many horizontally weakly conformal proper biharmonic maps including biharmonic Riemannian submersions Theorem \ref{MT2} and Remark \ref{RM3}.

\section{Characterizations of generalized harmonic morphisms into Euclidean spaces}

\begin{definition}
A map $\phi: (M^m, g)\longrightarrow (N^n,h)$ between Riemannian  manifolds is called {\bf a generalized harmonic morphism}, if for any harmonic function $f:N^n\supseteq U\longrightarrow \r$ with $\phi^{-1}(U)=V$ non-empty, the function
$f\circ\phi: V\longrightarrow \r$ is a biharmonic function on $V\subseteq M$.
\end{definition}
\begin{remark}
It is clear from the definitions of harmonic morphism, biharmonic morphism, and \gh\,  that any harmonic morphism is a \gh\,  and any biharmonic morphism is also a \gh. So we have the following inclusion relations:
$$\{ {\rm harmonic\; morphisms}\}\;\subset\; \{ {\rm generalized\; harmonic\; morphisms}\}, {\rm and}$$
$$\{ {\rm biharmonic\; morphisms}\}\;\subset\;\{ {\rm generalized\; harmonic\; morphisms}\}. $$
\end{remark}

Now we are ready to prove the following characterizations of the generalized harmonic morphisms into Euclidean spaces.
\begin{theorem}\label{MT}
Let  $\phi: (M^m, g)\longrightarrow \r^n$  with $\phi(x)=(\phi^1(x),\phi^2(x),\cdots, \phi^n(x))$ be a map from a Riemannian manifold into a Euclidean space, then the following statements are equivalent:
\begin{itemize}
\item[(i)]  $\phi$ is a generalized harmonic morphism,
\item[(ii)] $\phi$ is a horizontally weakly conformal biharmonic map and $(\phi^{\alpha}+i\phi^{\beta})^2: (M^m, g)\longrightarrow \mathbb{C}$ is also a biharmonic map for any $\alpha\ne \beta=1,2,\cdots, n$,
\item[(iii)] there exists a function $\lambda: M \longrightarrow [0, \infty)$ such that
\begin{eqnarray}\label{CE}
 \Delta^{2} (f\circ\phi)  &=& \lambda^4(\Delta^2f)\circ\phi+2[\lambda^2\Delta \phi^{\alpha} +g (\nabla\lambda^2,\nabla\phi^{\alpha})](\partial_{\alpha}\Delta f)\circ \phi\\\notag
&&+[\Delta \lambda^2 +2g(\nabla\phi^1,\nabla\Delta\phi^1)+(\Delta\phi^1)^2](\Delta f)\circ \phi
\end{eqnarray}
\end{itemize}
\end{theorem}
\begin{proof}
For a map $\phi: (M^m, g)\longrightarrow (N^n, h)$ between Riemannian manifolds and a function $f$ on $N$, we have, by Lemma 2.5 in \cite{Ou1}, the following bi-Laplacian formula
\begin{eqnarray}\notag
 \Delta^{2} (f\circ\phi)  &=& ( f_{\alpha\beta\gamma\delta}\circ\phi )\left[ g(\nabla\phi^{\alpha}, \nabla\phi^{\beta}) g(\nabla\phi^{\gamma}, \nabla\phi^{\delta}) \right]  \\\label{GDO}
&&+( f_{\alpha\beta\gamma}\circ\phi )\left[ g(\nabla\phi^{\alpha}, \nabla\phi^{\beta}) \Delta\phi^{\gamma}+g(\nabla\phi^{\beta}, \nabla\phi^{\gamma}) \Delta\phi^{\alpha}\right. \\\notag
&&\left.+2g(\nabla g(\nabla\phi^{\alpha}, \nabla \phi^{\beta}), \nabla\phi^{\gamma})\right]\\\notag
&&+(f_{\alpha\beta}\circ\phi )\left[ \Delta g(\nabla\phi^{\alpha}, \nabla\phi^{\beta})+ 2 g(\nabla\phi^{\beta}, \nabla\Delta\phi^{\alpha}) +\Delta\phi^{\alpha}\Delta\phi^{\beta}\right]\\\notag
&&+(f_{\alpha}\circ\phi )\Delta^2 \phi^{\alpha},
\end{eqnarray}
where Einstein summation convention has been used and $f_{\alpha}, f_{\alpha\beta}$ etc denote the partial derivatives $\frac{\partial f}{\partial y^{\alpha}}, \frac{\partial^2 f}{\partial y^{\alpha}\partial y^{\beta}}$ respectively.
Using Lemma 2.4 in \cite{Ou1}, we can rewrite formula (\ref{GDO}) as
\begin{eqnarray}\notag
 \Delta^{2} (f\circ\phi)  &=& \sum_{\alpha=1}^n( f_{\alpha\alpha\alpha\alpha}\circ\phi )|\nabla\phi^{\alpha}|^4+4 \sum_{1\le\alpha\ne \beta\le n}( f_{\alpha\alpha\alpha\beta}\circ\phi )|\nabla\phi^{\alpha}|^2g(\nabla\phi^{\alpha}, \nabla\phi^{\beta})  \\\notag
&&+\sum_{1\le \alpha\ne\beta\le n}( f_{\alpha\beta\alpha\beta}\circ\phi )[2|\nabla\phi^{\alpha}|^2|\nabla\phi^{\beta}|^2+4 (g(\nabla\phi^{\alpha}, \nabla\phi^{\beta}))^2]\\\notag
&&\sum_{1\le \alpha\ne\beta\ne\gamma\le n}( f_{\alpha\alpha\beta\gamma}\circ\phi )\left[ 4g(\nabla\phi^{\alpha}, \nabla\phi^{\alpha}) g(\nabla\phi^{\beta}, \nabla\phi^{\gamma}) \right.\\\notag
&&\left.+ 8g(\nabla\phi^{\alpha}, \nabla\phi^{\beta}) g(\nabla\phi^{\alpha}, \nabla\phi^{\gamma})\right] \\\notag
&&+\sum_{1\le \alpha\ne\beta\ne\gamma\ne\delta\le n}( f_{\alpha\beta\gamma\delta}\circ\phi )\left[ g(\nabla\phi^{\alpha}, \nabla\phi^{\beta}) g(\nabla\phi^{\gamma}, \nabla\phi^{\delta}) \right]  \\\notag
&&+\sum_{1\le \alpha\le n}( f_{\alpha\alpha\alpha}\circ\phi )\left[ 2|\nabla\phi^{\alpha}|^2 \Delta\phi^{\alpha}+2g(\nabla|\nabla\phi^{\alpha}|^2, \nabla\phi^{\alpha})\right]  \\\notag
&&+\sum_{1\le \alpha\ne \beta\le n}( f_{\alpha\alpha\beta}\circ\phi )\left[ 2|\nabla\phi^{\alpha}|^2 \Delta\phi^{\beta}+4g(\nabla\phi^{\alpha}, \nabla\phi^{\beta}) \Delta\phi^{\alpha} \right.\\\label{MF}
&&\left.+2g(\nabla|\nabla\phi^{\alpha}|^2, \nabla\phi^{\beta})+4g(\nabla g(\nabla\phi^{\alpha}, \nabla \phi^{\beta}), \nabla \phi^{\alpha})\right]  \\\notag
&&+\sum_{1\le \alpha\ne\beta\ne\gamma\le n}( f_{\alpha\beta\gamma}\circ\phi )\left[ g(\nabla\phi^{\alpha}, \nabla\phi^{\beta}) \Delta\phi^{\gamma}+g(\nabla\phi^{\beta}, \nabla\phi^{\gamma}) \Delta\phi^{\alpha}\right. \\\notag
&&\left.+2g(\nabla g(\nabla\phi^{\alpha}, \nabla \phi^{\beta}), \nabla\phi^{\gamma})\right]\\\notag
&&+\sum_{1\le \alpha\le n}(f_{\alpha\alpha}\circ\phi )\left[\frac{1}{2} \Delta^2 (\phi^{\alpha}\phi^{\alpha})-\phi^{\alpha}\Delta^2 \phi^{\alpha}\right]\\\notag
&&+\sum_{1\le\alpha<\beta\le n}(f_{\alpha\beta}\circ\phi )\left[ \Delta^2 (\phi^{\alpha}\phi^{\beta})-\phi^{\alpha}\Delta^2 \phi^{\beta}-\phi^{\beta}\Delta^2 \phi^{\alpha}\right]\\\notag
&&+\sum_{1\le \alpha\le n}(f_{\alpha}\circ\phi )\Delta^2 \phi^{\alpha}.
\end{eqnarray}

Let $\phi: (M^m, g)\longrightarrow \r^n$ be a generalized harmonic morphism, and let  $\{ y^{\alpha}\}$ be the standard Cartesian Coordinates on the Euclidean space $\r^n$, then $f(y)=y^{\alpha}$ $ \alpha=1,2, \cdots, n$ is a harmonic function for $ \alpha=1,2, \cdots, n$. Substituting $f(y)=y^{\alpha}$ into Equation (\ref{MF}) we obtain 
\begin{equation}\label{Bi}
\Delta^2 \phi^{\alpha}=0\;\; {\rm  for\; all \; } \alpha=1,2, \cdots, n,
\end{equation}
 so the map $\phi: (M^m, g)\longrightarrow \r^n$ is a biharmonic map.\\

By choosing  harmonic functions $f(y)=(y^{\alpha})^2-(y^{\beta})^2$ and $f(y)=y^{\alpha}y^{\beta}, \;\alpha\ne \beta$ using Equation (\ref{MF}) respectively,  we have
\begin{eqnarray}\notag
&& \Delta^2 (\phi^{\alpha}\phi^{\alpha})-2\phi^{\alpha}\Delta^2 \phi^{\alpha}-[\Delta^2 (\phi^{\beta}\phi^{\beta})-2\phi^{\beta}\Delta^2 \phi^{\beta}]=0,\\\notag
&&\Delta^2 (\phi^{\alpha}\phi^{\beta})-\phi^{\alpha}\Delta^2 \phi^{\beta}-\phi^{\beta}\Delta^2 \phi^{\alpha}=0.
\end{eqnarray}

This, together with Equation (\ref{Bi}), implies that
\begin{equation}\label{Sbi}
\begin{cases}
\Delta^2 (\phi^{\alpha}\phi^{\alpha})-\Delta^2 (\phi^{\beta}\phi^{\beta})=0,\\
\Delta^2 (\phi^{\alpha}\phi^{\beta})=0,
\end{cases}
\end{equation}
which means exactly that the map $(\phi^{\alpha}+i\phi^{\beta})^2: (M^m, g)\longrightarrow \r^2$ is a biharmonic map for any $\alpha\ne \beta$.

Using Equations (\ref{Bi}) and (\ref{Sbi}) we can rewrite Equation (\ref{MF}) as
\begin{eqnarray}\notag
 \Delta^{2} (f\circ\phi)  &=& \sum_{\alpha=1}^n( f_{\alpha\alpha\alpha\alpha}\circ\phi )|\nabla\phi^{\alpha}|^4+4 \sum_{1\le\alpha\ne \beta\le n}( f_{\alpha\alpha\alpha\beta}\circ\phi )|\nabla\phi^{\alpha}|^2g(\nabla\phi^{\alpha}, \nabla\phi^{\beta})  \\\label{MF2}
&&+\sum_{1\le \alpha\ne\beta\le n}( f_{\alpha\beta\alpha\beta}\circ\phi )[2|\nabla\phi^{\alpha}|^2|\nabla\phi^{\beta}|^2+4 (g(\nabla\phi^{\alpha}, \nabla\phi^{\beta}))^2]\\\notag
&&\sum_{1\le \alpha\ne\beta\ne\gamma\le n}( f_{\alpha\alpha\beta\gamma}\circ\phi )\left[ 4g(\nabla\phi^{\alpha}, \nabla\phi^{\alpha}) g(\nabla\phi^{\beta}, \nabla\phi^{\gamma}) \right.\\\notag
	&&\left.+ 8g(\nabla\phi^{\alpha}, \nabla\phi^{\beta}) g(\nabla\phi^{\alpha}, \nabla\phi^{\gamma})\right] \\\notag
&&+\sum_{1\le \alpha\ne\beta\ne\gamma\ne\delta\le n}( f_{\alpha\beta\gamma\delta}\circ\phi )\left[ g(\nabla\phi^{\alpha}, \nabla\phi^{\beta}) g(\nabla\phi^{\gamma}, \nabla\phi^{\delta}) \right]  \\\notag
&&+\sum_{1\le \alpha\le n}( f_{\alpha\alpha\alpha}\circ\phi )B_{\alpha\alpha\alpha}+\sum_{1\le \alpha\ne \beta\le n}( f_{\beta\beta\alpha}\circ\phi )B_{\beta\beta\alpha}\\\notag
&&+\sum_{1\le \alpha\ne\beta\ne\gamma\le n}( f_{\alpha\beta\gamma}\circ\phi )B_{\alpha\beta\gamma}  +\frac{1}{2} [(\Delta f)\circ\phi ] \Delta^2 (\phi^{1}\phi^{1})
\end{eqnarray}
for any function  $f$ on $\r^n$, where
\begin{eqnarray}
B_{\alpha\alpha\alpha}&= &2|\nabla\phi^{\alpha}|^2 \Delta\phi^{\alpha}+2g(\nabla|\nabla\phi^{\alpha}|^2, \nabla\phi^{\alpha})  \\
 B_{\beta\beta\alpha} &=& 2|\nabla\phi^{\beta}|^2 \Delta\phi^{\alpha}+4g(\nabla\phi^{\alpha}, \nabla\phi^{\beta}) \Delta\phi^{\beta} +2g(\nabla|\nabla\phi^{\beta}|^2, \nabla\phi^{\alpha})\\\notag
&&+4g(\nabla g(\nabla\phi^{\alpha}, \nabla \phi^{\beta}), \nabla \phi^{\beta}), \;\;\;1\le \alpha\ne\beta\le n.\\
B_{\alpha\beta\gamma} &=& g(\nabla\phi^{\alpha}, \nabla\phi^{\beta}) \Delta\phi^{\gamma}+g(\nabla\phi^{\beta}, \nabla\phi^{\gamma}) \Delta\phi^{\alpha} \\\notag &&+2g(\nabla g(\nabla\phi^{\alpha}, \nabla \phi^{\beta}),
 \nabla\phi^{\gamma}), \;\; 1\le \alpha\ne\beta\ne\gamma\le n.
\end{eqnarray}

By substituting the harmonic function $f(y)=y^{\alpha}y^{\beta}y^{\gamma}$ with $\alpha\ne\beta\ne\gamma$ into Equation (\ref{MF2}) we have 
\begin{equation}\label{GD1}
B_{\alpha\beta\gamma}=0, \;\;{\rm  for}\;\; 1\le \alpha\ne\beta\ne\gamma\le n.
\end{equation} 

On the other hand, by substituting the harmonic function $f(y)=(y^{\alpha})^3-3y^{\alpha}(y^{\beta})^2$  into Equation (\ref{MF2}) we have,
\begin{eqnarray}\label{GD2}
 B_{\alpha\alpha\alpha}=B_{\beta\beta\alpha}, \;\;{\rm for}\; 1\le \alpha\ne \beta \le n.
\end{eqnarray}
Substituting (\ref{GD1}) and (\ref{GD2}) into Equation (\ref{MF2}) we have
\begin{eqnarray}\notag
 \Delta^{2} (f\circ\phi)  &=& \sum_{\alpha=1}^n( f_{\alpha\alpha\alpha\alpha}\circ\phi )|\nabla\phi^{\alpha}|^4+4 \sum_{1\le\alpha\ne \beta\le n}( f_{\alpha\alpha\alpha\beta}\circ\phi )|\nabla\phi^{\alpha}|^2g(\nabla\phi^{\alpha}, \nabla\phi^{\beta})  \\\label{MF3}
&&\sum_{1\le \alpha\ne\beta\ne\gamma\le n}( f_{\alpha\alpha\beta\gamma}\circ\phi )\left[ 4g(\nabla\phi^{\alpha}, \nabla\phi^{\alpha}) g(\nabla\phi^{\beta}, \nabla\phi^{\gamma}) \right.\\\notag
	&&\left.+ 8g(\nabla\phi^{\alpha}, \nabla\phi^{\beta}) g(\nabla\phi^{\alpha}, \nabla\phi^{\gamma})\right] \\\notag
&&+\sum_{1\le \alpha\ne\beta\ne\gamma\ne\delta\le n}( f_{\alpha\beta\gamma\delta}\circ\phi )\left[ g(\nabla\phi^{\alpha}, \nabla\phi^{\beta}) g(\nabla\phi^{\gamma}, \nabla\phi^{\delta}) \right]  \\\notag
&&+\sum_{1\le \alpha\ne\beta\le n}( f_{\alpha\beta\alpha\beta}\circ\phi )[2|\nabla\phi^{\alpha}|^2|\nabla\phi^{\beta}|^2+4 (g(\nabla\phi^{\alpha}, \nabla\phi^{\beta}))^2]\\\notag
&&+\sum_{ \alpha=1}^n ( \frac{\partial}{\partial y^{\alpha}}\Delta f)\circ\phi )B_{\alpha\alpha\alpha} +\frac{1}{2} [(\Delta f)\circ\phi ] \Delta^2 (\phi^{1}\phi^{1})
\end{eqnarray}
for any function  $f$ on $\r^n$.\\

Choose the harmonic function $f(y)= (y^{\alpha})^3(y^{\beta})-(y^{\beta})^3y^{\alpha}$ and substitute it into Equation (\ref{MF3}), we have
\begin{equation}\label{GD3}
24 (|\nabla \phi^{\alpha}|^2-|\nabla \phi^{\beta}|^2)g(\nabla\phi^{\alpha}, \nabla\phi^{\beta}))=0, \;\;{\rm for}\; 1\le \alpha\ne \beta \le n.
\end{equation}

On the other hand, choosing the harmonic function $f(y)=(y^{\alpha})^4-6 (y^{\alpha})^2(y^{\beta})^2+(y^{\beta})^4$ and substituting it into Equation (\ref{MF3}) yields
\begin{equation}\label{GD4}
24 [(|\nabla \phi^{\alpha}|^2-|\nabla \phi^{\beta}|^2)^2-4(g(\nabla\phi^{\alpha}, \nabla\phi^{\beta})))^2]=0, \;\;{\rm for}\; 1\le \alpha\ne \beta \le n.
\end{equation}

Combining (\ref{GD3}) and (\ref{GD4}) we have 
\begin{equation}\label{HWC}
|\nabla \phi^{\alpha}|^2=|\nabla \phi^{\beta}|^2,\;\; g(\nabla\phi^{\alpha}, \nabla\phi^{\beta})=0, \;\;{\rm for}\; 1\le \alpha\ne \beta \le n
\end{equation}
which means exactly the map $\phi: (M^m, g)\longrightarrow \r^n$ is horizontally weakly conformal.\\

This, together with (\ref{Bi}) and (\ref{Sbi}), completes the proof that a \gh\, $\phi: (M^m, g)\longrightarrow \r^n$ is a horizontally weakly conformal biharmonic map with an additional condition that the map $(\phi^{\alpha}+i\phi^{\beta})^2: (M^m, g)\longrightarrow \mathbb{C}$ is also a biharmonic map for any $\alpha\ne \beta=1,2,\cdots, n$.\\

Conversely, if $\phi: (M^m, g)\longrightarrow \r^n$ is a horizontally weakly conformal biharmonic map with an additional condition that the map $(\phi^{\alpha}+i\phi^{\beta})^2: (M^m, g)\longrightarrow \mathbb{C}$ is also a biharmonic map for any $\alpha\ne \beta=1,2,\cdots, n$, we use (\ref{Bi}), (\ref{Sbi}), and (\ref{HWC}) to rewrite (\ref{MF}) as

\begin{eqnarray}\notag
 \Delta^{2} (f\circ\phi)  &=& \sum_{\alpha=1}^n( f_{\alpha\alpha\alpha\alpha}\circ\phi )|\nabla\phi^{\alpha}|^4  +\sum_{1\le \alpha\ne\beta\le n}( f_{\alpha\beta\alpha\beta}\circ\phi )[2|\nabla\phi^{\alpha}|^2|\nabla\phi^{\beta}|^2]\\\notag
&&+\sum_{1\le \alpha\le n}( f_{\alpha\alpha\alpha}\circ\phi )\left[ 2|\nabla\phi^{\alpha}|^2 \Delta\phi^{\alpha}+2g(\nabla|\nabla\phi^{\alpha}|^2, \nabla\phi^{\alpha})\right]  \\\notag
&&+\sum_{1\le \alpha\ne \beta\le n}( f_{\beta\beta\alpha}\circ\phi )\left[ 2|\nabla\phi^{\beta}|^2 \Delta\phi^{\alpha} +2g(\nabla|\nabla\phi^{\beta}|^2, \nabla\phi^{\alpha})\right] \\\notag
&&+\frac{1}{2} [(\Delta f)\circ\phi ] \Delta^2 (\phi^{1}\phi^{1})\\\label{ME}
&=& \lambda^4(\Delta^2f)\circ\phi+2[\lambda^2\Delta \phi^{\alpha} +g(\nabla\lambda^2,\nabla\phi^{\alpha})](\partial_{\alpha}\Delta f)\circ \phi\\\notag
&&+[\Delta \lambda^2 +2g(\nabla\phi^1,\nabla\Delta\phi^1)+(\Delta\phi^1)^2](\Delta f)\circ \phi,
\end{eqnarray}
where  in obtaining the last equality we have used the horizontal weak conformality $\lambda^2=|\nabla\phi^{\alpha}|^2=|\nabla\phi^{\beta}|^2$ and the identity (see Lemma 2.4 in \cite{Ou1})
\begin{eqnarray}\notag
\Delta^2 (\phi^{\alpha}\phi^{\alpha})&=&\Delta^2 (\phi^{1}\phi^{1})\\\notag
&=& 2\Delta g(\nabla\phi^1, \nabla\phi^1)+4g(\nabla\Delta\phi^1, \nabla\phi^1)+2\phi^1\Delta^2\phi^1+2(\Delta\phi^1)\\\notag
&=&2\Delta \lambda^2 +4g( \nabla\phi^1,\nabla\Delta\phi^1)+2(\Delta\phi^1)^2.
\end{eqnarray}

It follows from Equation (\ref{ME}) that the map $\phi: (M^m, g)\longrightarrow \r^n$ pulls back any harmonic function to a biharmonic function and hence it is a \gh. This completes the proof of the equivalence of Statements (i) and (ii).\\

Finally, the equivalence to Statements (i) and (iii) follows clearly from the definition of a \gh\; and Equation (\ref{CE}). Thus, we obtain the theorem.
\end{proof}

\begin{example}\label{EX1}
Let $M^4=\r^4\setminus\{0\times \r\}$ be the open submanifolds of $\r^4$ with the standard Euclidean metric. Then, the map $\phi:M^4\longrightarrow \r^2$ with $\phi(x_1,\cdots, x_4)=(\sqrt{x_1^2+x_2^2+x_3^2\,}, x_4)$ is a generalized harmonic morphism which is not a harmonic morphism. More precisely, $\phi$ is a Riemannian submersion which is also non-harmonic biharmonic map with additional property that $(\phi(x))^2=(\sqrt{x_1^2+x_2^2+x_3^2\,}+ix_4)^2$ is also a biharmonic map $\r^4\supset M^4\longrightarrow \mathbb{C}$. In fact, we can easily check that
\begin{itemize}
\item[(a)] $|\nabla\phi^1|=|\nabla\phi^2|=1$ and $\langle \nabla\phi^1, \nabla\phi^2\rangle=0$, so $\phi$ is a Riemannian submersion;
\item[(b)] $\Delta \phi^1=\frac{2}{\sqrt{x_1^2+x_2^2+x_3^2\,}},\;\; \Delta^2 \phi^1=0$, and   $\Delta \phi^2= \Delta^2 \phi^2=0$. So $\phi$ is a non-harmonic biharmonic map;
\item[(c)] the map 
\begin{eqnarray}
(\phi(x))^2&=&(\sqrt{x_1^2+x_2^2+x_3^2\,}+ix_4)^2\\\notag
&=&x_1^2+x_2^2+x_3^2-x_4^2+i(2x_4\sqrt{x_1^2+x_2^2+x_3^2}\,)
\end{eqnarray}
 is also a biharmonic map $\r^4\supset M^4\longrightarrow \mathbb{C}$.
\end{itemize}
\end{example}
\begin{remark}
 We would like to point out that there are examples of horizontally conformal biharmonic maps which are not a generalized harmonic morphism, for instance, we know from \cite{BFO} that the map $\phi:\r^3\longrightarrow \r^2$ with 
 \begin{equation}
 \phi(x_1,x_2,x_3)=\left(\frac{(1-\frac{1}{2}|x|^2)x_2+\sqrt{2}x_1x_3}{x_2^2+x_3^2}, \frac{(1-\frac{1}{2}|x|^2)x_3-\sqrt{2}x_1x_2}{x_2^2+x_3^2}\right)
 \end{equation}
is a horizontal weakly conformal biharmonic map. However, we can check that  $(\phi)^2: \r^3\longrightarrow \mathbb{C}$ is not biharmonic, so it is not a generalized harmonic morphism.
\end{remark}

Notice that the concept of harmonic morphisms traced back to  Jacobi who in 1848 tried to solve the Laplace equation 
\begin{equation}
\Delta \phi \equiv \sum_{i=1}^3\frac{\partial^2\phi}{\partial^2x_i}=0
\end{equation}
for $\phi:\r^3\supseteq U\longrightarrow \mathbb{C}$ under the additional condition that $f\circ\phi$ is harmonic for any (holomorphic) complex analytic function $f:\mathbb{C}\supseteq V\longrightarrow \mathbb{C}$.\\

Since the real and imaginary parts of any holomorphic function are harmonic and since the Laplacian is a linear operator on functions, the additional condition  that Jacobi imposed is equivalent to the map $\phi$ pulling back harmonic functions to harmonic functions. With this viewpoint, we can easily have 
 \begin{theorem}\cite{BW1}
 Let $\phi:\r^3\supseteq U\longrightarrow \mathbb{C}$ be a harmonic function. Then $f\circ\phi$ is harmonic for any (holomorphic) complex analytic function $f:\mathbb{C}\supseteq V\longrightarrow \mathbb{C}$ if and only if $\phi$ is a solution of
 \begin{eqnarray} \label{HWC}
\sum_{i=1}^3\left( \frac{\partial \phi}{\partial x_i} \right)^2=0.
\end{eqnarray}
 \end{theorem}
 Note that equation (\ref{HWC}) means exactly the map $\phi:\r^3\supseteq U\longrightarrow \mathbb{C}$ is horizontally weakly conformal.\\
 
Now, following Jacobi's idea and using Theorem \ref{MT}, we have
 \begin{proposition}\label{R3}
 Let $\phi:\r^3\supseteq U\longrightarrow \mathbb{C}$ be a biharmonic function. Then $f\circ\phi$ is biharmonic for any (holomorphic) complex analytic function $f:\mathbb{C}\supseteq V\longrightarrow \mathbb{C}$ if and only if $\phi$ is a solution of
 \begin{eqnarray} \label{HWC1}
\sum_{i=1}^3\left( \frac{\partial \phi}{\partial x_i} \right)^2=0
\end{eqnarray}
and $(\phi)^2: \r^3\supseteq U\longrightarrow \mathbb{C}$ is also a biharmonic function. 
 \end{proposition}
 
 Using Theorem \ref{MT} and the fact that a horizontally weakly conformal map between Riemann surfaces are simply weakly conformal maps which are harmonic, we have
 \begin{corollary}\label{CR3}
 (a) Any \gh\, $\phi:(M^2, g) \longrightarrow \r^2$ is a harmonic morphism, and \\
 (b) A map $\phi:\r^3\supseteq U\longrightarrow \mathbb{C}$ is \gh\, if and only if it is a horizontally weakly conformal map and both $\phi$ and $(\phi)^2$ are biharmonic.
 \end{corollary}

 \begin{remark}
 It was noted in \cite{BW1} that a map $\phi:\r^3\supseteq U\longrightarrow \mathbb{C}$ is a harmonic morphism if and only if both $\phi$ and $(\phi)^2$ are harmonic maps. So the statement of (b) in Corollary \ref{CR3} gives us another viewpoint to see how \gh s generalize the notion of harmonic morphisms.
 \end{remark}

Finally, we close this section by the following classifications of \gh s which are also Riemannian submersions.

\begin{corollary}
 If a Riemannian submersion $\phi:\r^3 \longrightarrow \r^2$ pulls back harmonic functions to biharmonic functions, then it is an orthogonal projection up to an isometry of the domain and/or the target.\\
(ii) There is no  Riemannian submersion $\phi:H^3 \longrightarrow \r^2$ pulls back harmonic functions to biharmonic functions.
\end{corollary}
\begin{proof}
If a Riemannian submersion  $\phi:\r^3 \longrightarrow \r^2$ or  $\phi:H^3 \longrightarrow \r^2$ pulls back harmonic functions to biharmonic functions, then it is a \gh, which, by Theorem \ref{MT}, has to be a biharmonic map. It was proved in \cite{WO} that  any  biharmonic Riemannian submersion  $\phi:\r^3 \longrightarrow \r^2$ has to be an orthogonal projection up to an isometry of the domain and/or the target, and that there is no biharmonic Riemannian submersion from $H^3$. Thus, we obtain the corollary.
\end{proof}

A family of proper biharmonic Riemannian submersions was constructed in \cite{LO2}. Later it was proved  in \cite{WO} that  the Riemannian submersion given by the projection of the warped product
\begin{align}\notag
\pi  : ( \mathbb{R}^2 \times \mathbb{R} , dx^2 +
dy^2+\beta^{-2}(x,y) dz^2) &\to (\mathbb{R}^2 ,dx^2 + dy^2) \\\notag
\pi(x,y,z) =(x,y)
\end{align}
 is biharmonic if and only if 
\begin{equation}
\Delta u=\Delta v=0,\;\;{\rm  where} \;\;u=(\ln \beta)_x,\; v=(\ln \beta)_y.
\end{equation}
By assuming $\beta=e^{\int\varphi(x)dx}\,e^{\int\phi(y)dy},$ the authors in \cite{WO} were able to solve for some special solutions for $\beta({x,y})$.\\
\indent Our next theorem shows that, if we add an additional condition that the square map $(\pi)^2$ is also biharmonic (which makes $\pi$  a \gh ),  then the system of equations can be solved completely. Thus, we obtain a complete classification of \gh s among the projections of the warped product, which also provides two families of infinitely many examples of proper biharmonic Rimannian submersions from a warped product space.
\begin{theorem}\label{MT2}
 The Riemannian submersion defined by the projection of the warped product
\begin{align}\notag
\pi  : ( \mathbb{R}^2 \times \mathbb{R} , dx^2 +
dy^2+\beta^{-2}(x,y) dz^2) &\to (\mathbb{R}^2 ,dx^2 + dy^2) \\\notag
\pi(x,y,z) =(x,y)
\end{align}
 is a \gh \; if and only if 
  \begin{eqnarray}
\beta(x, y)=C(x+C_1y+C_2)^{-2},\;\;{\rm or}\;\;\beta(x, y)=C(C_1x+y+C_2)^{-2}
 \end{eqnarray}
 for any constants $C>0, C_1, C_2$.
\end{theorem}

\begin{proof}

As in \cite{WO}, we choose an orthonormal frame
$\{e_1=\frac{\partial}{\partial x},\;e_2= \frac{\partial}{\partial
y}, \;e_3=\beta \frac{\partial}{\partial z}\}$  on  $M=( \mathbb{R}^2
\times \mathbb{R} , dx^2 + dy^2+\beta^{-2}(x,y) dz^2)$ adapted to
the Riemannian submersion $\pi$ with $ d\pi(e_i)=\varepsilon_i, i=1,
2$ and $e_3$ being vertical, where
$\varepsilon_1=\frac{\partial}{\partial x},\;\varepsilon_2=
\frac{\partial}{\partial y},\;$ form an orthonormal frame on the
base space $(\mathbb{R}^2 , dx^2 + dy^2)$. Then, as it was computed in \cite{WO}, we have
\begin{align}\notag
[e_1,e_3]=u e_3,\;[e_2,e_3]=v e_3 ,\;\; [e_1,e_2]=0,
\end{align}
where $u=(\ln \beta)_{x}, v=(\ln \beta)_{y}\;$.\\
The integrability data of the Riemannian submersion
$\pi$ are given by
\begin{equation}\notag
\begin{array}{lll}
f_1=f_2=\sigma=0,\;\kappa_1=u,\;\kappa_2=v.
\end{array}
\end{equation}
 and hence we have
\begin{eqnarray}\label{NA}
&&\nabla_{e_{1}} e_{1}=\nabla_{e_{2}} e_{2}=0,\;\; \nabla_{e_{3}} e_{3}=u e_{1}+v
e_2,
\end{eqnarray}
 where $\nabla$  denote the Levi-Civita connection of  warped product metric on the total space. The tension field of the Riemannian  submersion $\pi$ given by
\begin{equation}\label{RSB6}
\tau(\pi)=\nabla^{\pi}_{e_i}d\pi(e_i)-d\pi(\nabla^{M}_{e_i}e_i)=-d\pi(\nabla^{M}_{e_3}e_3)
=-u\varepsilon_1-v
\varepsilon_2.
\end{equation}

 It follows from Example 1 in \cite{WO} that the Riemannian submersion
$\pi$ is biharmonic if and only if
\begin{equation}\label{D4}
\Delta_{M} u=0,\;\;\;\Delta_{M} v=0.
\end{equation}

It follows from Theorem \ref{MT} that the projection $\pi$ (being a Riemannian submersion)  is a \gh\,  if and only if both $\pi$ and its square map $(\pi(x,y,z))^2=(x+iy)^2=x^2-y^2+2ixy$ are biharmonic maps. These are equivalent to $\beta(x, y)$ solving Equation (\ref{D4}) and 
\begin{equation}\label{D5}
\begin{cases}
\Delta_M^2(x^2-y^2)=0, \\
\Delta_M^2(xy)=0.
\end{cases}
\end{equation}
 A straightforward computation using (\ref{NA}) and (\ref{D4}) gives

 \begin{eqnarray}\label{LU}
 \Delta_{M} u&=&u_{xx}+u_{yy}-uu_x-vu_y,\\\label{LV}
\Delta_{M} v&=&v_{xx}+v_{yy}-uv_x-vv_y\\\notag
\Delta_M(x^2-y^2)&=&-2xu+2yv, \\\label{X}
\Delta^2_M(x^2-y^2)&=&-2(-u^2+2u_x)+2(-v^2+2v_y),\\\notag
\Delta_M(xy)&=&-yu-xv,\\\label{XY}
\Delta_M^2(xy)&=&2(uv-u_y-v_x).
\end{eqnarray}
It follows from (\ref{LU}), (\ref{LV}), (\ref{X}), and (\ref{XY}) that  the map $\pi$ is a \gh\, if and only if
 \begin{eqnarray}\label{1}
u_{xx}+u_{yy}-uu_x-vu_y=0,\\\label{2}
v_{xx}+v_{yy}-uv_x-vv_y=0,\\\label{3}
u^2-2u_x-v^2+2v_y=0,\\\label{4}
uv-u_y-v_x=0.
\end{eqnarray}
To solve this system, we first notice that
\begin{equation}\label{B1}
u_y=(\ln \beta)_{xy}=v_x,
\end{equation}
which, together with (\ref{4}), gives
\begin{equation}\label{B2}
u_y=v_x=\frac{1}{2}uv.
\end{equation}
On the other hand, we differentiate both sides of (\ref{3}) with respect to $x$ and both sides of (\ref{4}) with respect to $y$ to have, respectively
\begin{eqnarray}
2uu_x-2u_{xx}-2vv_x+2v_{xy}=0\\
u_yv+uv_y-u_{yy}-v_{xy}=0.
\end{eqnarray}
Multiplying $2$ to the second and adding the result to the first of the above equations, and using (\ref{1}) we have
\begin{equation}\label{B3}
uv_y-vv_x=0.
\end{equation}
Similarly, differentiating both sides of (\ref{3}) with respect to $y$ and both sides of (\ref{4}) with respect to $x$ gives, respectively
\begin{eqnarray}
2uu_y-2u_{xy}-2vv_y+2v_{yy}=0\\
u_xv+uv_x-u_{xy}-v_{xx}=0.
\end{eqnarray}
Multiplying $-2$ to the second and adding the result to the first of the above equations, and using (\ref{2}) we have
\begin{equation}\label{B4}
uu_y-vu_x=0.
\end{equation}
Now we are ready to solve the system (\ref{1})-(\ref{4}) by considering the following two cases.\\
\indent {\bf Case I:} $uv=0.$  In this case, $u=u(x), v=0$ or $u=0, v=v(y)$ and the system reduces to
\begin{equation}
\begin{cases}
v=0,\\u'-\frac{1}{2}u^2=0,\\(u'-\frac{1}{2}u^2)'=0,\\
\end{cases}
{\rm or}\;\;\;\;\;\;
\begin{cases}
u=0,\\ v'-\frac{1}{2}v^2=0,\\(v'-\frac{1}{2}v^2)'=0.
\end{cases}
\end{equation}
Solving these systems, we have
\begin{equation}\label{Sp}
\beta(x, y)=C(x+C_2)^{-2},\;\;{\rm or}\;\;\beta(x, y)=C(y+C_2)^{-2}.
\end{equation}
\indent {\bf Case II:} $uv\ne0.$ First, we substitute (\ref{B2}) into (\ref{B3}) and (\ref{B4})  to have, respectively,
\begin{eqnarray}
uv_y-vu_y=0, \;\;\; {\rm and}\;\; uv_x-u_xv=0.
\end{eqnarray}
Since $uv\ne0$, we divide both sides of the above equations by $uv$ and rewrite them as
\begin{eqnarray}
\left(\ln \frac{u}{v}\right)_y=0, \;\;\; {\rm and}\;\; \left(\ln \frac{u}{v}\right)_x=0.
\end{eqnarray}
From these we have
\begin{equation}
u=C_1v, \;\;\;{\rm or}\;\;\;v=C_1u
\end{equation}
for some constant $C_1\ne 0$. Substituting $u=C_1v$ into (\ref{B2}) we have
\begin{eqnarray}
\begin{cases}
v_x=\frac{1}{2}C_1v^2,\\
C_1v_y=\frac{1}{2}C_1v^2,
\end{cases}
\end{eqnarray}
which is equivalent to
\begin{eqnarray}
\begin{cases}
v_x=\frac{1}{2}C_1v^2,\\
v_y=\frac{1}{2}v^2.
\end{cases}
\end{eqnarray}
Solving this system we have 
\begin{equation}\label{S1}
u(x,y)=\frac{-2C_1}{C_1x+y+C_2},\;\;v(x,y)=\frac{-2}{C_1x+y+C_2}
\end{equation}
for some constant $C_2$.
Similarly, using $v=C_1u$ Equation (\ref{B2}) have another family of solutions given by
\begin{equation}\label{S2}
u(x,y)=\frac{-2}{x+C_1y+C_2},\;\;v(x,y)=\frac{-2C_1}{x+C_1y+C_2}
\end{equation}
for some constant $C_2$.
A straightforward checking shows that all the solutions in (\ref{S1}) and (\ref{S2}) also solve equations (\ref{1}), (\ref{2}) and (\ref{3}), so they are solutions of the \gh\,  equations.\\
\indent Substituting  $u(x,y), v(x,y)$ given in (\ref{S1}) and (\ref{S2}) into the equation $(\ln \beta)_x=u(x,y), (\ln \beta)_y=v(x,y)$ and solving the resulting equations, we have
\begin{eqnarray}
\beta(x, y)=C(x+C_1y+C_2)^{-2},\;\;{\rm or}\;\;\beta(x, y)=C(C_1x+y+C_2)^{-2}
 \end{eqnarray}
 for any constants $C>0, C_1\ne 0, C_2$. \\
 \indent Finally, notice that if we allow $C_1=0$, then the family of the solutions contains the special solutions given in (\ref{Sp}). Thus, we obtain the theorem.
\end{proof}
\begin{remark}\label{RM3}
(i) Notice that the tension field of the Riemannian submersions given in the theorem   is $\tau(\pi)=u d\pi(e_1)+v d\pi(e_)\ne 0$ for all the solutions given in (\ref{S1}) and (\ref{S2}), so the \gh s provided by our theorem are all proper biharmonic Riemannian submersions.\\
(ii) It is very interesting to note, as one can easily check, that the square map 
\begin{align}\notag
(\pi) ^2: ( \mathbb{R}^2 \times \mathbb{R} , dx^2 +
dy^2+\beta^{-2}(x,y) dz^2) &\to (\mathbb{R}^2 ,dx^2 + dy^2) \\\notag
(\pi)^2(x,y,z) =(x^2-y^2,2xy)
\end{align}
of the \gh s given in our theorem provides two families of infinitely many examples of conformal proper biharmonic submersions from a $3$-dimensional warped product manifolds.
\end{remark}

\section{Some constructions and more examples of \gh s}

In this section, we give two methods  of constructions to produce new \gh s from given ones. One is by composition and the other is by using a direct sum.

Unlike in the case of harmonic or biharmonic morphisms where we have the composition rules stating that the compositions of harmonic morphisms (respectively, biharmonic morphisms) are harmonic morphisms (respectively, biharmonic morphisms), the composition of \gh s are not expected to be \gh s. However, by the definitions of harmonic morphisms, biharmonic morphisms and the generalized harmonic morphisms, we do have the following composition rules which provide methods to produce new \gh s from  given ones by composing it  with a harmonic morphism or a biharmonic morphism.
\begin{corollary}\label{C1}
(I)  If $\phi: (M^m, g)\longrightarrow (N^n, h)$ is a generalized harmonic morphism and $\varphi: (N^n, h)\longrightarrow (Q^l, k)$ is a harmonic morphism, then $\varphi\circ \phi: (M^m, g)\longrightarrow (Q^l, k)$ is a generalized harmonic morphism.\\
(II) If $\phi: (M^m, g)\longrightarrow (N^n, h)$ is a generalized harmonic morphism and $\psi: (Q^l, k)\longrightarrow (M^n, g)$ is a biharmonic morphism, then $\phi\circ \psi: (Q^l, k)\longrightarrow (N^n, h)$ is a generalized harmonic morphism.
\end{corollary}

\begin{remark}
Notice that the concept of \gh s into Riemann surfaces are well defined. Recall that a Riemann surface is an orientable surface with a conformal class of Riemannian metrics. By (I) of Corollary \ref{C1}, if $\phi: (M^m, g)\longrightarrow (N^2,h)$ is a generalized harmonic morphism, then $\phi: (M^m, g)\longrightarrow (N^2,\sigma^2h)$ is also a generalized harmonic morphism since the map can be viewed as the composition $(M^m, g)\,\underrightarrow{\;\;\;\phi\;\;} \, (N^2,h)\,\underrightarrow{\;\;\;id\;\;} \,(N^2,\sigma^2h)$ and the identity map is a harmonic morphism. It follows that the concept of generalized harmonic morphisms into Riemann surfaces is well defined.
\end{remark}

\begin{example}\label{EX10}
Let  $\phi:\r^4\longrightarrow \r^2$ with $\phi(x_1,\cdots, x_4)=(\sqrt{x_1^2+x_2^2+x_3^2\,}, x_4)$ be the generalized harmonic morphism  given in Example \ref{EX1}. Let $\sigma^{-1}: \r^2\longrightarrow S^2$ be the inverse of the stereographic projection. Then, by Corollary \ref{C1}, the composition $\sigma^{-1}\circ\phi:\r^4\longrightarrow S^2$ is a generalized harmonic morphism.  In particular, it is a  proper biharmonic map from $\r^4$ into $2$-sphere. 
\end{example}

\begin{example}\label{EX2}
Let $\psi: \r^4\longrightarrow \r^4$ be the inversion $\psi(x)=\frac{x}{|x|^2}$ which is a biharmonic morphism by \cite{LO2}. Let $\phi:\r^4\longrightarrow \r^2$ with $\phi(x_1,\cdots, x_4)=(\sqrt{x_1^2+x_2^2+x_3^2\,}, x_4)$ be the  generalized harmonic morphism given in Example \ref{EX1}. Then, by ( II ) of Corollary \ref{C1}, the composition $(\phi\circ \psi):\r^4\longrightarrow \r^2$ with 
\begin{equation}
(\phi\circ \psi)(x)=\left(\frac{\sqrt{x_1^2+x_2^2+x_3^2\,}}{|x|^2}, \frac{x_4}{|x|^2}\right)
\end{equation}
 is a \gh\; with dilation $\lambda=1/|x|^2$. In particular, it is a horizontally weakly conformal biharmonic map. Also, by Theorem \ref{MT}, the map 
 \begin{equation}
 \left(\frac{\sqrt{x_1^2+x_2^2+x_3^2\,}}{|x|^2}+i \frac{x_4}{|x|^2}\right)^2= \left(\frac{x_1^2+x_2^2+x_3^2-x_4^2}{|x|^4}+2i \frac{x_4\sqrt{x_1^2+x_2^2+x_3^2\,}}{|x|^4}\right)
 \end{equation}
  is also a biharmonic map $\r^4\longrightarrow \mathbb{C}$.
\end{example}

\begin{example}\label{EX3}
The map $ \varphi: \r^4\longrightarrow \r^3$ given by
\begin{equation}\label{Hopf}
\varphi(x)=\left(\frac{x_1^2+x_2^2-x_3^2-x_4^2}{|x|^2},\; \frac{2x_1x_3-2x_2x_4}{|x|^2},\;\frac{2x_1x_4+2x_2x_3}{|x|^2}\right).
\end{equation}
is a \gh. To see this, let $\psi: \r^4\longrightarrow \r^4$ be the inversion $\psi(x)=\frac{x}{|x|^2}$ which is a biharmonic morphism by \cite{LO2}. Let $\phi:\r^4\longrightarrow \r^3$ be the Hopf map defined by
\begin{eqnarray}
\phi(x_1,\cdots, x_4)=(x_1^2+x_2^2-x_3^2-x_4^2, \;2x_1x_3-2x_2x_4,\; 2x_1x_4+2x_2x_3)
\end{eqnarray}
 Then, the map $\varphi$ is exactly the composition $\varphi=\phi\circ \psi$. It follows from Corollary \ref{C1} that   $\varphi=\phi\circ \psi$ is a \gh\,  since it is a composition of a biharmonic morphism and a (generalized) harmonic morphism.
We notice that  the map $\varphi$ is a non-harmonic  horizontally weakly conformal biharmonic map with dilation $\lambda=2/|x|^{3/2}$  is known in \cite{BFO}. 
\end{example}

\begin{example}\label{EX4}
Let $\varphi:\r^4\longrightarrow \r^3$ denote the \gh\;  in Example \ref{EX3} defined by (\ref{Hopf}). Then, by composing $\varphi$ with an orthogonal projection $\r^3\longrightarrow \r^2\equiv\mathbb{C}$, we have three different \gh s \\$\varphi^{\alpha}+i \varphi^{\beta}: \r^4\longrightarrow \mathbb{C}$ for any $\alpha\ne \beta=1,2,3$. 
\end{example}

Recall that a direct sum of two maps $\psi: (M^m, g)\longrightarrow \r^k,\;\varphi: (N^n, h)\longrightarrow \r^k$ is defined (see \cite{Ou0}) to be the map $\psi\oplus\varphi: (M^m\times N^n, g\times h) \longrightarrow \r^k$ with
\begin{equation}
(\psi\oplus\varphi)(x, y)=\psi(x)+\varphi(y).
\end{equation}

The direct sums of maps have been used ( see \cite{Ou0}, \cite{BW1} and  \cite{Ou2}) to construct examples of harmonic morphisms and biharmonic maps. For example, it follows from \cite{Ou0} and \cite{Ou2} that the direct sum of horizontally weakly conformal biharmonic maps is again a horizontally weakly conformal  biharmonic map. However, one can easily check that,  in general, the direct sum of two \gh s is not a \gh. Nevertheless, we have the following proposition which gives a method to produce many examples of \gh s.

\begin{proposition}
Let  $\psi: (M^m, g)\longrightarrow \r^k$  be a \gh  and and $\varphi: (N^n, h)\longrightarrow \r^k$ a harmonic morphism. Then, their direct sum $\psi\oplus\varphi: (M^m\times N^n, G=g\times h) \longrightarrow \r^k$ is  a \gh. In particular, $\psi\oplus\varphi$ is a horizontally weakly conformal proper biharmonic map.
\end{proposition}
\begin{proof}
By Theorem \ref{MT}, it suffices to prove the statement for the case of $\r^k=\mathbb{C}$. Suppose that  $\psi: (M^m, g)\longrightarrow \mathbb{C}$ is a \gh, and $\varphi: (N^n, h)\longrightarrow\mathbb{C}$ is a harmonic morphism. It was proved in \cite{Ou0} that the direct sum of two harmonic morphisms is a harmonic morphism, in particular, the direct sum of two horizontally weakly conformal maps with dilations $\lambda_1(x)$ and $\lambda_2(y)$ is again a horizontally weakly conformal map with dilation $\lambda(x, y)=\lambda_1(x)+\lambda_2(y)$. On the other hand, it was proved in \cite{Ou2} that the direct sum of two biharmonic maps is a biharmonic map. It follows from Theorem \ref{MT} that to prove the direct sum  $\psi\oplus\varphi: (M^m\times N^n, G=g\times h) \longrightarrow \mathbb{C}$ is a \gh, it is enough to check that the map  $(\psi\oplus\varphi)^2: (M^m\times N^n, G=g\times h) \longrightarrow \mathbb{C}$ is a biharmonic map. Notice that
\begin{eqnarray}\notag
(\psi\oplus\varphi)^2(x, y)= (\psi(x)+\varphi(y))^2=(\psi(x))^2+(\varphi(y))^2+2\psi(x)\varphi(y).
\end{eqnarray}
On the other hand, since the product metric $G=g\times h$ is used on the product manifold  $M\times N$, we have $\Delta^2_G=(\Delta_g+\Delta_h)^2$. It follows that
\begin{eqnarray}\notag
\Delta^2_G(\psi\oplus\varphi)^2 &=& \Delta^2_g(\psi(x))^2+\Delta^2_h(\varphi(y))^2\\\notag
&&+2[\varphi(y)\Delta^2_g\psi(x) +\psi(x)\Delta^2_h\varphi(y)+2(\Delta_g\psi(x))\Delta_h\varphi(y)]=0,
\end{eqnarray}
where in obtaining the last equality, we have used the assumptions that $\psi$ is a \gh\, and $\varphi$ is a harmonic morphism.
\end{proof}

\begin{example}
Let  $\phi:\r^4\longrightarrow \r^2$ with $\phi(x_1,\cdots, x_4)=(\sqrt{x_1^2+x_2^2+x_3^2\,}, x_4)$ be the generalized harmonic morphism  given in Example \ref{EX1}, and let $\varphi :\mathbb{C}\longrightarrow \mathbb{C}$ be any holomorphic function which is well known to be a harmonic morphism. then $\psi\oplus\varphi: \r^4\times \mathbb{C} \longrightarrow \mathbb{C}$ with  $(\psi\oplus\varphi)(x_1,\cdots, x_4, z)=\sqrt{x_1^2+x_2^2+x_3^2\,}+i\,x_4+\varphi(z)$ is a \gh, and in particular, it is a horizontally weakly conformal proper biharmonic map.
\end{example}

\section{Maps between Riemannian manifolds that pull back biharmonic functions to harmonic functions}

The generalized harmonic morphism we studied in this paper are maps between Riemannian manifolds that pull back harmonic functions to biharmonic functions. One may wonder if there is any map between Riemannian manifolds that pulls back biharmonic functions to harmonic functions. In this concluding section, we prove that any such map is a constant map. More precisely, we have
\begin{theorem}
There exists no non-constant map between between Riemannian manifolds
that pulls back germs of biharmonic functions to germs of harmonic
functions.
\end{theorem}
\begin{proof}
Let $\varphi : (M^m ,g) \longrightarrow (N^n ,h)$  be a  map between
Riemannian manifolds. If $\varphi$ pulls back germs of
 biharmonic functions to germs of harmonic functions, then it must
 pull back germs of harmonic functions (which are a special subset of biharmonic functions) to germs of harmonic
 functions. Therefore, $\varphi$ is a harmonic morphism by definition (see
\cite{Fu}, \cite{Is}, also,
 \cite{BW1}). It follows from \cite{Fu} and \cite{Is} that
\begin{equation}\label{L1}
\Delta_{M}(f\circ \varphi)=\lambda^2(\Delta_{N} f)\circ \varphi
\end{equation}
for any (locally defined) function $f$ on $N$. Take a $f$ to be
quasi-harmonic function (see e.g.,  \cite{NS}) i.e., a special biharmonic
function satisfying $\Delta_{N} f=C$, a nonzero constant, then $f\circ
\varphi$ would be harmonic by the assumption that $\varphi$ pulls
back a biharmonic  function to a harmonic function. It follows from
(\ref{L1}) that $0=\Delta_{M}( f\circ \varphi)=\lambda^2(\Delta
_{N}f)\circ \varphi=C\lambda^2$, which is possible only if the map
is a constant map. Thus, we obtain the theorem.
\end{proof}
\begin{ack}
Both authors would like to thank Paul Baird for some comments that help to improve the manuscript.
\end{ack}

\end{document}